\documentclass[preprint,10pt]{elsarticle}

\usepackage{hyperref}
\usepackage{amsmath,amssymb,amsthm}
\usepackage{dsfont}  % for \mathds{1}, the indicator function

\makeatletter
\def\moverlay{\mathpalette\mov@rlay}
\def\mov@rlay#1#2{\leavevmode\vtop{%
   \baselineskip\z@skip \lineskiplimit-\maxdimen
   \ialign{\hfil$\m@th#1##$\hfil\cr#2\crcr}}}
\newcommand{\charfusion}[3][\mathord]{
    #1{\ifx#1\mathop\vphantom{#2}\fi
        \mathpalette\mov@rlay{#2\cr#3}
      }
    \ifx#1\mathop\expandafter\displaylimits\fi}
\makeatother
\newcommand{\cupdot}{\charfusion[\mathbin]{\cup}{\cdot}}
\newcommand{\bigcupdot}{\charfusion[\mathop]{\bigcup}{\cdot}}

\newcommand{\R}{\mathbb R}
\newcommand{\Q}{\mathbb Q}
\newcommand{\Z}{\mathbb{Z}}
\newcommand{\N}{\mathbb{N}}

\def\bz{ \boldsymbol{z}}

\newtheorem{theorem}{Theorem}

\newtheorem{lemma}[theorem]{Lemma}
\newtheorem{proposition}[theorem]{Proposition}
\newtheorem*{definition}{Definition}        % * added!
\newtheorem{Conjecture}{Conjecture}
\newtheorem{Prob}{Problem}

\makeatletter
\newcommand{\neutralize}[1]{\expandafter\let\csname c@#1\endcsname\count@}
\makeatother

\newenvironment{Conjectureprime}[1]
   {\renewcommand{\theConjecture}{\ref{#1}$'$}%
   \neutralize{Conjecture}\phantomsection
   \begin{Conjecture}}
  {\end{Conjecture}}

\usepackage{color}

\begin{document}

\begin{frontmatter}

%% Title, authors and addresses

%% use the tnoteref command within \title for footnotes;
%% use the tnotetext command for theassociated footnote;
%% use the fnref command within \author or \address for footnotes;
%% use the fntext command for theassociated footnote;
%% use the corref command within \author for corresponding author footnotes;
%% use the cortext command for theassociated footnote;
%% use the ead command for the email address,
%% and the form \ead[url] for the home page:
%% \title{Title\tnoteref{label1}}
%% \tnotetext[label1]{}
%% \author{Name\corref{cor1}\fnref{label2}}
%% \ead{email address}
%% \ead[url]{home page}
%% \fntext[label2]{}
%% \cortext[cor1]{}
%% \affiliation{organization={},
%%             addressline={},
%%             city={},
%%             postcode={},
%%             state={},
%%             country={}}
%% \fntext[label3]{}

\title{A note on exponential Riesz bases}
%%%\title{Exponential bases with hierarchical structure}

%% use optional labels to link authors explicitly to addresses:
%% \author[label1,label2]{}
%% \affiliation[label1]{organization={},
%%             addressline={},
%%             city={},
%%             postcode={},
%%             state={},
%%             country={}}
%%
%% \affiliation[label2]{organization={},
%%             addressline={},
%%             city={},
%%             postcode={},
%%             state={},
%%             country={}}

\author{Andrei Caragea}
\ead{andrei.caragea@gmail.com}

\author{Dae Gwan Lee\corref{cor1}}
\ead{daegwans@gmail.com}
\cortext[cor1]{Corresponding author}

\address{Mathematisch-Geographische Fakult\"at, Katholische Universit\"at Eichst\"att-Ingolstadt, \\
85071 Eichst{\"a}tt, Germany}

%%%\address{Lehrstuhl f\"ur Mathematik -- Wissenschaftliches Rechnen, Mathematisch-Geographische Fakult\"at, Katholische Universit\"at Eichst\"att-Ingolstadt, 85071 Eichst{\"a}tt, Germany}
%%%\address{Mathematisch-Geographische Fakult\"at, Katholische Universit\"at Eichst\"att-Ingolstadt, \\
%%%85071 Eichst{\"a}tt, Germany\\
%%%Email: andrei.caragea@ku.de, daegwans@gmail.com}

%%%\affiliation{organization={Mathematisch-Geographische Fakult\"at, Katholische Universit\"at Eichst\"att-Ingolstadt},%Department and Organization
%%%            addressline={A},
%%%            city={Eichst{\"a}tt},
%%%            postcode={85071},
%%%            state={Bavaria},
%%%            country={Germany}}

\begin{abstract}
We prove that if $I_\ell = [a_\ell,b_\ell)$, $\ell=1,\ldots,L$, are disjoint intervals in $[0,1)$ with the property that the numbers $1, a_1, \ldots, a_L, b_1, \ldots, b_L$ are linearly independent over $\Q$,
then there exist pairwise disjoint sets $\Lambda_\ell \subset \Z$, $\ell=1, \ldots, L$,
such that for every $J \subset \{ 1, \ldots , L \}$, the system $\{e^{2\pi i \lambda x} : \lambda\in \cup_{\ell \in J} \, \Lambda_\ell \}$ is a Riesz basis for $L^2 ( \cup_{\ell \in J} \, I_\ell)$.
Also, we show that for any disjoint intervals $I_\ell$, $\ell=1, \ldots, L$, contained in $[1,N)$ with $N \in \N$, the orthonormal basis $\{e^{2\pi i n x} : n \in \Z \}$ of $L^2[0,1)$ can be complemented by a Riesz basis $\{e^{2\pi i \lambda x} : \lambda\in\Lambda\}$ for $L^2(\cup_{\ell=1}^L \, I_{\ell})$ with some set $\Lambda \subset (\frac{1}{N} \Z) \backslash \Z$, in the sense that their union $\{e^{2\pi i \lambda x} : \lambda\in \Z \cup \Lambda\}$ is a Riesz basis for $L^2 ( [0,1) \cup I_1 \cup \cdots \cup I_L )$.
\end{abstract}

%%%Graphical abstract
%\begin{graphicalabstract}
%%\includegraphics{grabs}
%\end{graphicalabstract}

%%%Research highlights
%\begin{highlights}
%\item Research highlight 1
%\item Research highlight 2
%\end{highlights}

\begin{keyword}
%%% keywords here, in the form: keyword \sep keyword
exponential bases \sep
Riesz bases \sep
hierarchical structure \sep
finite union of intervals \sep
Kronecker--Weyl equidistribution along the primes
%%% hierarchy \sep
%%% complementability \sep
%%% Riesz sequences \sep
%%% frames \sep

%% PACS codes here, in the form: \PACS code \sep code

%% MSC codes here, in the form: \MSC code \sep code
%% or \MSC[2008] code \sep code (2000 is the default)
\MSC 42C15
\end{keyword}

\end{frontmatter}

%% \linenumbers

%% main text
%\section{}
%\label{}

\section{Introduction and Main Results}
\label{sec:intro}

In 1995, Seip \cite{Se95} showed that if $S$ is an interval contained in $[0,1)$, then there exists a set $\Lambda \subset \Z$ such that $E(\Lambda) := \{e^{2\pi i \lambda x} : \lambda\in\Lambda\}$ is a Riesz basis for $L^2(S)$.
Since then, there have been various attempts towards finding/characterizing the sets $S$ that admit a Riesz spectrum, see e.g., \cite{AAC15,CC18,CHM21,DL19,GL14,Ko15}.
A significant breakthrough was made by Kozma and Nitzan \cite{KN15} who proved that if
$[a_\ell,b_\ell)$, $\ell=1, \ldots, L$, are disjoint intervals contained in $[0,1)$, then there exists a set $\Lambda \subset \Z$ such that $E(\Lambda)$ is a Riesz basis for $L^2 ( \cup_{\ell=1}^L [a_\ell,b_\ell) )$.
Recently, Pfander, Revay and Walnut \cite{PRW21} showed that if the intervals $[a_\ell,b_\ell)$ form a partition of $[0,1)$, then the set of integers $\Z$ can be partitioned into some sets $\Lambda_\ell$, $\ell=1, \ldots, L$, such that for each $\ell$, the system $E(\Lambda_\ell)$ is a Riesz basis for $L^2[a_\ell,b_\ell)$, and moreover $E( \cup_{\ell \in J} \, \Lambda_\ell )$ is a Riesz basis for $L^2 ( \cup_{\ell \in J} \, S_\ell )$ whenever $J \subset \{ 1, \ldots , L \}$ is a consecutive index set (see \cite[Theorems 1 and 2]{PRW21}).
We would like to point out that up to date, the existence of exponential Riesz bases is known only for several classes of sets $S \subset \R$. Recently, Kozma, Nitzan and Olevskii \cite{KNO21} constructed a bounded measurable set $S \subset \R$ such that the space $L^2(S)$ has no exponential Riesz basis.
For an overview of the known results on exponential Riesz bases, we refer to \cite[Section 1]{Le21}.

We are interested in the following two problems:

\begin{Prob}[Hierarchical structured exponential Riesz bases]
\label{P1}
Given a family of disjoint sets $S_1, S_2 , \ldots, S_L \subset [0,1)$ with positive measure, can we find disjoint sets $\Lambda_1, \Lambda_2, \ldots, \Lambda_L \subset \Z$ such that for every $J \subset \{ 1, \ldots , L \}$, the system $E( \cup_{\ell \in J} \, \Lambda_\ell )$ is a Riesz basis for $L^2 ( \cup_{\ell \in J} \, S_\ell )$?
\end{Prob}

\begin{Prob}[Complementability of exponential Riesz bases]
\label{P2}
Let $\Lambda \subset \R$ be a discrete set and let $S \subset \R$ be a finite positive measure set such that $E(\Lambda)$ is a Riesz basis for $L^2(S)$.
Given a finite positive measure set $S' \subset \R \backslash S$, can we find a discrete set $\Lambda' \subset \R \backslash \Lambda$ such that
\begin{itemize}
\item
$E(\Lambda')$ is a Riesz basis for $L^2 (S')$, and

\item
$E(\Lambda \cup \Lambda')$ is a Riesz basis for $L^2 (S \cup S')$?
\end{itemize}
\end{Prob}

The second problem is closely related to the first, as it deals with the case $L=2$ under the assumption that the sets $S_1$ and $\Lambda_1$ are already fixed.

Considering the result of Kozma, Nitzan and Olevskii \cite{KNO21}, it is necessary to restrict the sets $S_\ell$, $S$ and $S'$ to certain classes of sets.
In this paper, we will address the above problems in the case that $S_\ell$, $S$ and $S'$ are intervals or finite unions of intervals.

Our first main result answers Problem~\ref{P1} in the affirmative
%%% when $S_\ell$ are disjoint separated intervals in $[0,1)$ whose endpoints are linearly independent over $\Q$, more precisely, when $S_\ell = [a_\ell,b_\ell)$ for $\ell=1,\ldots,L$ and
%%% when $S_\ell \subset [0,1)$ are disjoint separated intervals given by $S_\ell = [a_\ell,b_\ell)$ for $\ell=1,\ldots,L$, where the numbers $1, a_1, \ldots, a_L, b_1, \ldots, b_L$ are linearly independent over $\Q$,
when $S_\ell = [a_\ell,b_\ell)$, $\ell=1,\ldots,L$, are disjoint intervals in $[0,1)$ with the property that the numbers $1, a_1, \ldots, a_L, b_1, \ldots, b_L$ are linearly independent over $\Q$,
which means that having
$q + q_1 a_1 + \ldots + q_L a_L + q_1' b_1 + \ldots + q_L' b_L = 0$ for some $q , q_\ell, q_\ell' \in \Q$ implies $q = q_\ell = q_\ell' = 0$ for all $\ell$.
The result is motivated by \cite[p.279, Claim 2]{KN15}.
%%% with irrational endpoints that are linearly independent over $\Q$; the result is motivated by \cite[p.279, Claim 2]{KN15}.

\begin{theorem}\label{thm:hierarchy-for-finite-intervals-Q-linearly-indep}
Let $0 < a_1 < b_1 < \cdots < a_L < b_L < 1$ with $L \in \N$.
Assume that the numbers $1, a_1, \ldots, a_L, b_1, \ldots, b_L$ are linearly independent over $\Q$.
%which means that having
%$q + q_1 a_1 + \ldots + q_L a_L + q_1' b_1 + \ldots + q_L' b_L = 0$ for some $q , q_\ell, q_\ell' \in \Q$ implies $q = q_\ell = q_\ell' = 0$ for all $\ell$.
There exist pairwise disjoint sets $\Lambda_\ell \subset \Z$, $\ell=1, \ldots, L$,
such that for every $J \subset \{ 1, \ldots , L \}$, the system $E( \cup_{\ell \in J} \, \Lambda_\ell )$ is a Riesz basis for $L^2 ( \cup_{\ell \in J} \, [a_\ell,b_\ell) )$.
\end{theorem}

Concerning Problem~\ref{P2}, we have the following result which builds on the fact that $E(\Z)$ is an orthonormal basis (thus, a Riesz basis) for $L^2[0,1)$.

\begin{theorem}\label{thm:unit-interval-complementable-with-finite-union-intervals}
Let $1 \leq a_1 < b_1 < a_2 < b_2 < \cdots < a_L < b_L \leq N$ with $L, N \in \N$.
There exists a set $\Lambda' \subset (\frac{1}{N} \Z) \backslash \Z$ such that
\begin{itemize}
\item
$E(\Lambda')$ is a Riesz basis for $L^2 ( \cup_{\ell =1}^L [a_\ell,b_\ell) )$, and

\item
$E(\Z \cup \Lambda')$ is a Riesz basis for $L^2 ( [0,1) \cup [a_1,b_1) \cup \cdots \cup [a_L,b_L) )$.
\end{itemize}
\end{theorem}

This theorem answers Problem~\ref{P2} in the affirmative when $S = [0,1)$, $\Lambda = \Z$, and $S'$ is a finite union of disjoint bounded intervals in $[1, \infty)$.

While the result of Pfander, Revay and Walnut \cite{PRW21} relies on Avdonin's theorem and the ergodic properties of a certain type of integer sequences,
Theorem \ref{thm:hierarchy-for-finite-intervals-Q-linearly-indep} is based on a refinement of the key lemma of \cite{KN15} for primes, together with Chebotar\"{e}v's theorem on roots of unity and the Kronecker--Weyl equidistribution theorem along the primes.

\subsection{Remarks}

We state the following conjecture which improves upon Theorem \ref{thm:hierarchy-for-finite-intervals-Q-linearly-indep}.

\begin{Conjecture}
\label{conj:full-hierarchy}
Let $[a_\ell,b_\ell)$, $\ell=1, \ldots, L$, be disjoint intervals contained in $[0,1)$, that is, $0 \leq a_1 < b_1 \leq a_2 < b_2 \leq \cdots \leq a_L < b_L \leq 1$.
There exist pairwise disjoint sets $\Lambda_\ell \subset \Z$, $\ell=1, \ldots, L$,
such that for every $J \subset \{ 1, \ldots , L \}$, the system $E( \cup_{\ell \in J} \, \Lambda_\ell )$ is a Riesz basis for $L^2 ( \cup_{\ell \in J} \, [a_\ell,b_\ell) )$.
\end{Conjecture}

%%% This conjecture generalizes Theorem \ref{thm:hierarchy-for-finite-intervals-Q-linearly-indep} by allowing for contiguous (non-separated) intervals in $[0,1)$, i.e., the cases where $b_\ell = a_{\ell+1}$ for some $\ell \in \{ 1,\ldots,L-1 \}$.
This conjecture generalizes Theorem \ref{thm:hierarchy-for-finite-intervals-Q-linearly-indep} by removing the $Q$-linear independence of the endpoints and by allowing for contiguous intervals in $[0,1)$, i.e., $b_\ell = a_{\ell+1}$ for some $\ell \in \{ 1,\ldots,L-1 \}$.
The conjecture can be easily reformulated as follows.

\begin{Conjectureprime}{conj:full-hierarchy}
\label{conj:full-hierarchy-prime}
Let $I_\ell$, $\ell=1, \ldots, L$, be intervals which form a partition of $[0,1)$.
There exists a partition $\Lambda_\ell$, $\ell=1, \ldots, L$, of $\Z$ such that for every $J \subset \{ 1, \ldots , L \}$, the system $E( \cup_{\ell \in J} \, \Lambda_\ell )$ is a Riesz basis for $L^2 ( \cup_{\ell \in J} \, I_\ell )$.
\end{Conjectureprime}

Indeed, Conjecture \ref{conj:full-hierarchy} obviously implies Conjecture \ref{conj:full-hierarchy-prime}, and the converse is seen by considering the partition of $[0,1)$ formed using the endpoints of all $[a_\ell,b_\ell)$.
Note that Conjecture \ref{conj:full-hierarchy-prime} generalizes the result of Pfander, Revay and Walnut \cite{PRW21} from consecutive index sets $J$ to arbitrary index sets $J$.

Lastly, we mention that both Problems~\ref{P1} and \ref{P2} remain open for more general classes of sets $S_\ell$, $S$ and $S'$.

\section{Preliminaries}
\label{sec:prelim}

\begin{definition}\label{def:HilbertSpace-Sequences}
A sequence $\{ f_n \}_{n\in\Z}$ in a separable Hilbert space $\mathcal{H}$ is called
\begin{itemize}
\item  a \emph{frame} for $\mathcal{H}$ (with frame bounds $A$ and $B$) if there are constants $0 < A \leq B < \infty$ such that
\[
A \, \| f \|^2
\;\leq\;
\sum_{n\in\Z} |\langle f , f_n\rangle|^{2}
\;\leq\;
B \, \| f \|^2
\quad \text{for all} \;\; f \in \mathcal{H};
\]

\item  a \emph{Riesz sequence} in $\mathcal{H}$ (with Riesz bounds $A$ and $B$) if there are constants $0 < A \leq B < \infty$ such that
\[
A \, \| c \|_{\ell_2}^2
\;\leq\;
\Big\| \sum_{n\in\Z} c_n \, f_n \Big\|^2
\;\leq\;
B \, \| c \|_{\ell_2}^2
\quad \text{for all} \;\; \{c_n\}_{n\in\Z} \in \ell_2 (\mathbb Z);
\]

\item  a \emph{Riesz basis} for $\mathcal{H}$ if it is a complete Riesz sequence in $\mathcal{H}$.
\end{itemize}
\end{definition}

It is well-known (see e.g., \cite[Proposition 3.7.3, Theorems 5.4.1 and 7.1.1]{Ch16} or \cite[Lemma 1]{KN15}) that a sequence in $\mathcal{H}$ is a Riesz basis if and only if it is both a frame and a Riesz sequence. Moreover in this case, the optimal frame bounds coincides with the optimal Riesz bounds.
It is worth noting that Riesz bases are equivalent to unconditional bases that are norm-bounded above and below \cite[Lemma 3.6.9]{Ch16}.
Since every exponential function has constant norm in $L^2(S)$ with $S \subset \R^d$,
namely $\| e^{2 \pi i \lambda \cdot (\cdot)} \|_{L^2(S)} = |S|^{1/2}$ for any $\lambda \in \R^d$, \emph{Riesz bases of exponentials} coincide with \emph{unconditional bases of  exponentials}.

\begin{proposition}[Proposition 2.1 in \cite{MM09}, Proposition 5.4 in \cite{BCMS19}]
\label{prop:Prop5-4-BCMS19}
Let $\{ e_n \}_{n \in I}$ be an orthonormal basis of a separable Hilbert space $\mathcal{H}$, where $I$ is a countable index set.
Let $P : \mathcal{H} \rightarrow \mathcal{M}$ be the orthogonal projection from $\mathcal{H}$ onto a closed subspace $\mathcal{M}$.
Let $J \subset I$, $J^c := I \backslash J$, and $0 < \alpha \leq 1$. The following are equivalent.
\begin{itemize}
\item[$\mathrm{(\romannumeral 1)}$] $\{ P e_n \}_{n \in J} \subset \mathcal{M}$ is a frame for $\mathcal{M}$ with optimal lower bound $\alpha$.

\item[$\mathrm{(\romannumeral 2)}$] $\{ P e_n \}_{n \in J^c} \subset \mathcal{M}$ is a Bessel sequence with optimal bound $1-\alpha$.

\item[$\mathrm{(\romannumeral 3)}$] $\{ (\mathrm{Id}-P) e_n \}_{n \in J^c} \subset \mathcal{M}^{\perp}$ is a Riesz sequence with optimal lower bound $\alpha$.
\end{itemize}
\end{proposition}

As a direct consequence of Proposition \ref{prop:Prop5-4-BCMS19}, we have that for a set $\Lambda \subset \Z$ and a measurable set $S \subset [0,1)$, the system $E(\Lambda)$ is a frame for $L^2(S)$ if and only if $E(\Z \backslash \Lambda)$ is a Riesz sequence in $L^2 ( [0,1)\backslash S )$.

\begin{lemma}\label{lem:RB-basic-operations}
Assume that $E(\Lambda)$ is a Riesz basis for $L^2(S)$ with bounds $0 < A \leq B < \infty$, where $\Lambda \subset \R^d$ is a discrete set and $S \subset \R^d$ is a measurable set. Then the following hold. \\
(a) For any $a,b \in \R^d$, the system $E(\Lambda+a)$ is a Riesz basis for $L^2(S+b)$ with bounds $A$ and $B$. \\
(b) For any $c > 0$, the system $E(c \Lambda)$ is a Riesz basis for $L^2(\frac{1}{c} S)$ with bounds $\frac{A}{c}$ and $\frac{B}{c}$.
\end{lemma}

Lemma \ref{lem:RB-basic-operations} remains valid if all the terms ``Riesz basis'' are replaced by ``Riesz sequence'' or by ``frame''.
A proof of Lemma \ref{lem:RB-basic-operations} can be found in \cite{Le21}.

For any $N \in \N$, a measurable set $S \subset [0,1)$, and $n = 1, \ldots, N$, we define
\begin{equation}\label{eqn:A-geq-n}
\begin{split}
A_{\geq n} &= A_{\geq n} (N,S)
:= \Big\{ t \in [0, \tfrac{1}{N}) : t+ \tfrac{k}{N} \in S \;\; \text{for at least $n$ values} \\
& \qquad\qquad\qquad\qquad\qquad\qquad\quad\;\; \text{of} \;\; k \in \{ 0 , 1, \ldots, N-1 \} \Big\}  .
\end{split}
\end{equation}

\begin{lemma}[Lemma 2 in \cite{KN15}]
\label{lem:KN15-Lemma2}
Let $N \in \N$ and let $S \subset [0,1)$ be a measurable set.
If there exist sets $\Lambda_1, \ldots, \Lambda_N \subset N\Z$ such that $E(\Lambda_n)$ is a Riesz basis for $L^2(A_{\geq n})$, then $E( \cup_{n=1}^N (\Lambda_n {+} n) )$ is a Riesz basis for $L^2 (S)$.
\end{lemma}

This lemma, which plays a central role in \cite{KN15}, combines Riesz bases by introducing \emph{consecutive} shift factors $n$ to the frequency sets $\Lambda_n$ and then taking their union $\cup_{n=1}^N (\Lambda_n {+} n)$.
For our purpose, we generalize the lemma to allow for \emph{arbitrary} shift factors when $N$ is prime.

\begin{lemma}\label{lem:KN15-Lemma2-variant-Nprime}
Let $N \in \N$ be a prime and let $S \subset [0,1)$ be a measurable set.
If there exist sets $\Lambda_1, \ldots, \Lambda_N \subset N\Z$ such that $E(\Lambda_n)$ is a Riesz basis (resp.~a frame, a Riesz sequence) for $L^2(A_{\geq n})$, then for every permutation $\{ j_n \}_{n=1}^N$ of $\{ 1, \ldots, N \}$ the system $E( \cup_{n=1}^N (\Lambda_n {+} j_n) )$ is a Riesz basis (resp.~a frame, a Riesz sequence) for $L^2 (S)$.
\end{lemma}

See \ref{appendix:proof-or-KN15-Lemma2-variant-Nprime} for a proof of Lemma \ref{lem:KN15-Lemma2-variant-Nprime}.

We will use the following notation throughout the proofs.
For $x \in \R$, we denote the fractional part of $x$ by $\{ x \}$, that is, $0 \leq \{ x \} := x - \lfloor x \rfloor < 1$, where $\lfloor x \rfloor$ is the greatest integer less than or equal to $x$.
Also, we adopt the convention that $[x,y) = \emptyset$ if $x = y \in \R$.

For the proof of Theorem \ref{thm:hierarchy-for-finite-intervals-Q-linearly-indep}, we will need the following version of the Kronecker--Weyl equidistribution theorem (see e.g., \cite[Theorem 443]{HW08} or \cite[p.48, Theorem 6.3 and Example 6.1]{KN74}) along the \emph{primes}. The one-dimensional case ($d=1$) was proved by Vinogradov \cite{Vi37} (see also \cite[p.22]{KN74}): if $a$ is an \emph{irrational} number, the sequence $\{ 2 a \}, \{ 3 a \}, \{ 5 a \}, \{ 7 a \}, \ldots$ is \emph{uniformly distributed} in $[0,1)$, meaning that for every interval $I \subset [0,1)$, the ratio of the numbers $\{ p a \}$ with prime $p \leq N$ that are contained in $I$, tends to $|I|$ as $N \rightarrow \infty$.
The notion of \emph{uniform distribution} is defined similarly in higher dimensions, see e.g., \cite[p.47, Definition 6.1]{KN74}.
As we could not find any reference for the multi-dimensional case, we include a short proof here.

\begin{proposition}[Kronecker--Weyl equidistribution along the primes]
\label{prop:multi-dim-Vinogradov}
Let $d \in \N$ and $a_1, \ldots, a_d \in \R$.
If the numbers $1, a_1, \ldots, a_d$ are linearly independent over $\Q$, which means that having $q + q_1 a_1 + \ldots + q_d \, a_d = 0$ for some $q,q_1,\ldots,q_d \in \Q$ implies $q = q_1 = \ldots = q_d = 0$, then the $d$-dimensional vectors
\[
\big( \{ p \, a_1 \} , \ldots, \{ p \, a_d \}  \big)
\quad \text{for} \;\; p \in \mathcal{P}
\]
are uniformly distributed in $[0,1)^d$, where $\mathcal{P} = \{ 2,3,5,7, \ldots \}$ is the set of primes.
\end{proposition}

\begin{proof}
For convenience, we denote the $n$-th prime by $p_n$, that is, $p_1 {=} 2$, $p_2 {=} 3$, $p_3 {=} 5$, $p_4 {=} 7$, and so on.
By Weyl's criterion (see e.g., \cite[p.48, Theorems 6.2 and 6.3]{KN74}), the claim is equivalent to having that for every $\bz = (z_1, \ldots, z_d) \in \Z^d \backslash \{ 0 \}$, the fractional part of $\langle \bz, ( p_n a_1 , \ldots, p_n a_d ) \rangle = z_1 \cdot p_n a_1 + \ldots + z_d \cdot p_n a_d = p_n \cdot (z_1 a_1 + \ldots + z_d \, a_d)$ for $n=1,2,\ldots$ are uniformly distributed in $[0,1)$.
Note that for any fixed $\bz = (z_1, \ldots, z_d) \in \Z^d \backslash \{ 0 \}$, the number $\widetilde{a} := z_1 a_1 + \ldots + z_d \, a_d$ is \emph{irrational} because $1, a_1, \ldots, a_d$ are linearly independent over $\Q$.
Hence, the result of Vinogradov \cite{Vi37} implies that the numbers $\{ p_n \widetilde{a} \}$, $n=1,2,\ldots$, are uniformly distributed in $[0,1)$, as desired.
\end{proof}

\section{Proof of Theorem \ref{thm:hierarchy-for-finite-intervals-Q-linearly-indep}}
\label{sec:proof-first-main-result}

Proposition \ref{prop:multi-dim-Vinogradov} implies that there exist infinitely many prime numbers $N \in \N$ satisfying
\begin{equation}\label{eqn:Naell-Nbell-intervals-nested}
\begin{split}
0 & \;<\; \{ N a_1 \} \;<\; \{ N a_2 \} \;<\; \ldots \;<\; \{ N a_{L-1} \} \;<\; \{ N a_L \} \\
&\quad \;<\;
\{ N b_L \} \;<\; \{ N b_{L-1} \} \;<\; \ldots \;<\; \{ N b_2 \} \;<\; \{ N b_1 \} \;<\; 1 .
\end{split}
\end{equation}
Among such numbers, choose a large $N \in \N$
so that every spacing between the numbers
$0 < a_1 < b_1 < a_2 < b_2 < \cdots < a_L < b_L < 1$
contains at least one of $\frac{k}{N}$, $k = 1, \ldots, N{-}1$, as an interior point (which clearly requires $2L + 1 \leq N$).
This ensures that with respect to the grid $\frac{1}{N} \Z$ the interval $[a_\ell,b_\ell)$ is partitioned into translates of
\begin{equation}\label{eqn:aell-bell-partitioned-into-three-cases}
\big[ \tfrac{\{ N a_\ell \}}{N} , \tfrac{1}{N} \big) ,
\quad
\big[ 0 , \tfrac{\{ N b_\ell \}}{N} \big) ,
\quad \text{and possibly some extra intervals} \;\; [0, \tfrac{1}{N}) ,
\end{equation}
and that the rightmost segment of $[a_\ell,b_\ell)$, corresponding to $[ 0 , \frac{\{ N b_\ell \}}{N} )$ in \eqref{eqn:aell-bell-partitioned-into-three-cases}, lies in a set $[ \frac{k}{N} , \frac{k+1}{N} )$ which does not intersect the next interval $[a_{\ell+1},b_{\ell+1})$.
Consequently, each of the sets $A_{\geq n} = A_{\geq n} (N,S)$, $n = 1, 2, \ldots, N$, is one of the form
\[
\emptyset ,
\quad
[0, \tfrac{1}{N}) ,
\quad \text{and} \quad
\big[ \tfrac{\{ N a_\ell \}}{N} , \tfrac{\{ N b_\ell \}}{N} \big)
\quad \text{for some} \;\;
\ell \in \{ 0 , 1, \ldots, N-1 \} .
\]
Note that due to \eqref{eqn:Naell-Nbell-intervals-nested}, the translates of $[ \frac{\{ N a_\ell \}}{N} , \frac{1}{N} )$ and $[ 0 , \frac{\{ N b_\ell \}}{N} )$ in \eqref{eqn:aell-bell-partitioned-into-three-cases} together contribute exactly $[ 0, \frac{1}{N} )$ and $[ \frac{\{ N a_\ell \}}{N} , \frac{\{ N b_\ell \}}{N} )$ to the family of sets $A_{\geq n}$.
The nested sets
\[
A_{\geq 1} \;\supset\; A_{\geq 2} \;\supset\; \cdots \;\supset\; A_{\geq N}
\]
are thus given by
\[
\begin{split}
&\overbrace{[0,\tfrac{1}{N}) = \cdots = [0,\tfrac{1}{N})}^{K}
\;\supset\;
\overbrace{\big[ \tfrac{\{ N a_1 \}}{N} , \tfrac{\{ N b_1 \}}{N} \big) \;\supset\;
\cdots \;\supset\;
\big[ \tfrac{\{ N a_L \}}{N} , \tfrac{\{ N b_L \}}{N} \big)}^{L} \\
&\;\supset\;
\overbrace{\emptyset = \cdots = \emptyset}^{N-K-L}
\qquad \text{for some integer} \;\; K \geq L .
\end{split}
\]
Let us associate each set $A_{\geq n}$ with the interval $[a_\ell,b_\ell)$ which it originates from.
The sets $[0,\frac{1}{N})$ can be associated with the intervals $[a_\ell,b_\ell)$ in various ways, but for convenience we will assume
\[
\begin{split}
&\underbrace{\overbrace{[0,\tfrac{1}{N}) = \cdots = [0,\tfrac{1}{N})}^{K_1}}_{\stackrel{\updownarrow}{[a_1,b_1)}}
 = \cdots =
\underbrace{\overbrace{[0,\tfrac{1}{N}) = \cdots = [0,\tfrac{1}{N})}^{K_L}}_{\stackrel{\updownarrow}{[a_L,b_L)}} \\
&\;\supset\;
\overbrace{\underbrace{\big[ \tfrac{\{ N a_1 \}}{N} , \tfrac{\{ N b_1 \}}{N} \big)}_{\stackrel{\updownarrow}{[a_1,b_1)}} \;\supset\;
\cdots \;\supset\;
\underbrace{\big[ \tfrac{\{ N a_L \}}{N} , \tfrac{\{ N b_L \}}{N} \big)}_{\stackrel{\updownarrow}{[a_L,b_L)}}
}^{L}
\;\supset\;
\overbrace{\emptyset = \cdots = \emptyset}^{N-K-L}
\end{split}
\]
where $K = \sum_{\ell=1}^L K_\ell$ with $K_\ell \in \N$ for all $\ell$.

\medskip

\noindent
\textbf{Step 1}. \emph{Construction of the sets $\Lambda_\ell \subset \Z$, $\ell=1, \ldots, L$}. \\
For each $n = 1, \ldots, N$, we apply the result of Seip \cite{Se95} (see the beginning of Section~\ref{sec:intro}) to obtain a set $\Lambda^{(n)} \subset N\Z$ such that $E(\Lambda^{(n)})$ is a Riesz basis for $L^2(A_{\geq n})$; it is easily seen that
\[
\Lambda^{(n)} =
\begin{cases}
N\Z & \text{for} \;\; 1 \leq n \leq K , \\
\subsetneq N\Z & \text{for} \;\; K{+}1 \leq n \leq K{+}L , \\
\emptyset & \text{for} \;\; K{+}L{+}1 \leq n \leq N .
\end{cases}
\]
Lemma \ref{lem:KN15-Lemma2} implies that $E ( \cup_{n=1}^N (\Lambda^{(n)} {+} n) )$ is a Riesz basis for $L^2(S)$.
For each $\ell=1, \ldots, L$, let $\Lambda_\ell$ be the union of $\Lambda^{(n)} {+} n$ over all $n$ such that $A_{\geq n}$ is associated with $[a_\ell,b_\ell)$, that is,
\[
\begin{split}
\Lambda_1 &:= \bigcupdot_{n \in \{ 1, \, 2 , \, \ldots, \, K_1 , \, K+1 \}} \, (\Lambda^{(n)} {+} n)
= \Big( \textstyle \bigcupdot_{n = 1}^{K_1} \, (N\Z {+} n) \Big) \;\bigcupdot\; (\Lambda^{(K+1)} {+} K {+} 1) , \\
\Lambda_2 &:= \bigcupdot_{n \in \{ K_1+1, \, K_1+2 , \, \ldots, \, K_1+K_2 , \, K+2 \}} \, (\Lambda^{(n)} {+} n) \\
\,&= \Big( \textstyle \bigcupdot_{n = K_1+1}^{K_1+K_2} \, (N\Z {+} n) \Big) \;\bigcupdot\; (\Lambda^{(K+2)} {+} K {+} 2) , \\
& \;\; \vdots \\
\Lambda_L &:= \bigcupdot_{n \in \{ K_1+\cdots+K_{L-1}+1, \, \ldots, \, K , \, K+L \}} \, (\Lambda^{(n)} {+} n) \\
&\,= \Big( \textstyle \bigcupdot_{n = K_1+\cdots+K_{L-1}+1}^{K} \, (N\Z {+} n) \Big) \;\bigcupdot\; (\Lambda^{(K+L)} {+} K {+} L) .
\end{split}
\]
Clearly, we have
\[
\bigcupdot_{\ell=1}^L \, \Lambda_\ell
= \bigcupdot_{n=1}^N \, (\Lambda^{(n)} {+} n)
\]
and thus, $E( \cup_{\ell=1}^L \, \Lambda_\ell )$ is a Riesz basis for $L^2(S)$.

\medskip

\noindent
\textbf{Step 2}.
For a subset $J \subset \{ 1, \ldots , L \}$, we set $\Lambda^J := \cup_{\ell \in J} \, \Lambda_\ell$ and $S^J := \cup_{\ell \in J} \, [a_\ell,b_\ell)$.
We claim that \emph{$E( \Lambda^J )$ is a Riesz basis for $L^2(S^J)$}.

First, note that the corresponding sets $A_{\geq n}^J := A_{\geq n} (N,S^J)$ for $n = 1, \ldots, N$ are again of the form
\[
\emptyset ,
\quad
[0, \tfrac{1}{N}) ,
\quad \text{or} \quad
\big[ \tfrac{\{ N a_\ell \}}{N} , \tfrac{\{ N a_\ell \}}{N} \big) \quad \text{for some} \;\;
\ell \in J .
\]
Denoting $J = \{ \ell_1, \ldots, \ell_M \}$ with $\ell_1 < \ldots < \ell_M$,
we see that the nested sets
\[
A_{\geq 1}^J \;\supset\; A_{\geq 2}^J \;\supset\; \cdots \;\supset\; A_{\geq N}^J
\]
are given by
\[
\begin{split}
&\underbrace{\overbrace{[0,\tfrac{1}{N}) = \cdots = [0,\tfrac{1}{N})}^{K_{\ell_1}}}_{\stackrel{\updownarrow}{\big[a_{\ell_1},b_{\ell_1}\big)}}
 = \cdots =
\underbrace{\overbrace{[0,\tfrac{1}{N}) = \cdots = [0,\tfrac{1}{N})}^{K_{\ell_M}}}_{\stackrel{\updownarrow}{\big[a_{\ell_M},b_{\ell_M}\big)}} \\
&\;\supset\;
\overbrace{\underbrace{\Big[ \tfrac{\{ N a_{\ell_1} \}}{N} , \tfrac{\{ N b_{\ell_1} \}}{N} \Big)}_{\stackrel{\updownarrow}{\big[a_{\ell_1},b_{\ell_1}\big)}}
\;\supset\; \cdots \;\supset\;
\underbrace{\Big[ \tfrac{\{ N a_{\ell_M} \}}{N} , \tfrac{\{ N b_{\ell_M} \}}{N} \Big)}_{\stackrel{\updownarrow}{\big[a_{\ell_M},b_{\ell_M}\big)}}
}^{M}
\;\supset\;
\overbrace{\emptyset = \cdots = \emptyset}^{N- K_J - M}
\end{split}
\]
where $K_J := \sum_{\ell \in J} K_{\ell}$.
Note that applying Lemma \ref{lem:KN15-Lemma2} directly to this setup will incur \emph{different shift factors} in the frequency sets.
Indeed, Lemma \ref{lem:KN15-Lemma2} implies that $E( \Lambda' )$ is a Riesz basis for $L^2(S^J)$, with
\begin{equation}\label{eqn:apply-Lemma2-directly-Lambda-prime}
\begin{split}
& \Lambda' := \Big( \textstyle \bigcupdot_{n = 1}^{K_J} \, (N\Z {+} n) \Big) \;\bigcupdot\; (\Lambda^{(K+\ell_1)} {+} K_J {+} 1)
\;\bigcupdot\; (\Lambda^{(K+\ell_2)} {+} K_J {+} 2) \\
&\qquad \qquad \;\bigcupdot\; \cdots
\;\bigcupdot\; (\Lambda^{(K+\ell_M)} {+} K_J {+} M) ,
\end{split}
\end{equation}
where $\Lambda^{(n)} \subset N\Z$ for $n = 1, \ldots, N$ are the sets defined in Step 1.
However, our goal is to show that $E( \cup_{m=1}^M \, \Lambda_{\ell_m} )$ is a Riesz basis for $L^2(S^J)$, where
\[
\begin{split}
\Lambda_{\ell_1} &:= \Big( \textstyle \bigcupdot_{n = K_1+K_2+\cdots+K_{\ell_1-1}+1}^{K_1+K_2+\cdots+K_{\ell_1}} \, (N\Z {+} n) \Big) \;\bigcupdot\; (\Lambda^{(K+\ell_1)} {+} K {+} \ell_1) , \\
\Lambda_{\ell_2} &:= \Big( \textstyle \bigcupdot_{n = K_1+K_2+\cdots+K_{\ell_2-1}+1}^{K_1+K_2+\cdots+K_{\ell_2}} \, (N\Z {+} n) \Big) \;\bigcupdot\; (\Lambda^{(K+\ell_2)} {+} K {+} \ell_2) , \\
& \;\; \vdots \\
\Lambda_{\ell_M} &:= \Big( \textstyle \bigcupdot_{n = K_1+K_2+\cdots+K_{\ell_M-1}+1}^{K_1+K_2+\cdots+K_{\ell_M}} \, (N\Z {+} n) \Big) \;\bigcupdot\; (\Lambda^{(K+\ell_M)} {+} K {+} \ell_M) .
\end{split}
\]

To show this, consider the $K_J {+} M$ sets
\[
\begin{split}
\Omega_1 &:= N\Z {+} (K_1{+}K_2{+}\cdots{+}K_{\ell_1-1}{+}1), \\
\vdots \;\; & \quad\qquad \vdots \\
\Omega_{K_{\ell_1}} &:= N\Z {+} (K_1{+}K_2{+}\cdots{+}K_{\ell_1}), \\
\Omega_{K_{\ell_1}+1} &:= N\Z {+} (K_1{+}K_2{+}\cdots{+}K_{\ell_2-1}{+}1), \\
\vdots \;\; & \quad\qquad \vdots \\
\Omega_{K_{\ell_1}+K_{\ell_2}} &:= N\Z {+} (K_1{+}K_2{+}\cdots{+}K_{\ell_2}), \\
\vdots \;\; & \quad\qquad \vdots \\
\Omega_{K_{\ell_1}+K_{\ell_2}+\cdots+K_{\ell_{M-1}}} &:= N\Z {+} ( K_1{+}K_2{+}\cdots{+}K_{\ell_M-1}{+}1 ), \\
\vdots \;\; & \quad\qquad \vdots \\
\Omega_{K_J} &:= N\Z {+} ( K_1{+}K_2{+}\cdots{+}K_{\ell_M}) = N\Z {+} K_J, \\
\Omega_{K_J+1} &:= \Lambda^{(K+\ell_1)} {+} (K {+} \ell_1), \\
\Omega_{K_J+2} &:= \Lambda^{(K+\ell_2)} {+} (K {+} \ell_2), \\
\vdots \;\; & \quad\qquad \vdots \\
\Omega_{K_J+M} &:= \Lambda^{(K+\ell_M)} {+} (K {+} \ell_M) ,
\end{split}
\]
which partitions $\cup_{m=1}^M \, \Lambda_{\ell_m}$, that is,
$\cupdot_{n=1}^{K_J +M} \, \Omega_n = \cup_{m=1}^M \, \Lambda_{\ell_m}$.
Here, the sets $\Omega_n$ are exactly ordered in the way that $E(\Omega_n)$ is a Riesz basis for $L^2(A_{\geq n}^J)$.
Note that while the $K_J {+} M$ components of $\Lambda'$ in \eqref{eqn:apply-Lemma2-directly-Lambda-prime} have \emph{consecutive shift factors}, namely from $1$ up to $K_J {+} M$, the shift factors associated with $\Omega_n$ are not consecutive in general.
However, since $N \in \N$ is prime, Lemma \ref{lem:KN15-Lemma2-variant-Nprime} implies that $E( \cup_{m=1}^M \, \Lambda_{\ell_m} ) = E( \cupdot_{n=1}^{K_J +M} \, \Omega_n )$ is a Riesz basis for $L^2(S^J)$.
This completes the proof.

\section{Proof of Theorem \ref{thm:unit-interval-complementable-with-finite-union-intervals}}
\label{sec:proof-second-main-result}

To prove Theorem \ref{thm:unit-interval-complementable-with-finite-union-intervals}, we will use Lemma \ref{lem:KN15-Lemma2} which is the key lemma of Kozma and Nitzan \cite{KN15}.
Note that by Lemma \ref{lem:RB-basic-operations}, one may replace the frequency set $\cup_{n=1}^N (\Lambda_n {+} n)$ in Lemma \ref{lem:KN15-Lemma2} by $\cup_{n=1}^N (\Lambda_n {+} n{-}1)$, while preserving the Riesz basis property.

We will first prove the case $L = 1$ and then extend the proof to the case $L \geq 2$.

\medskip

\noindent
\textbf{Case $L = 1$}. \
Given a set $V = [0,1) \cup [a,b) \subset [0,N)$ with $N \in \N$ and $1 \leq a < b \leq N$, let $S := \frac{1}{N} V = [0,\frac{1}{N}) \cup [\frac{a}{N},\frac{b}{N}) \subset [0,1)$.
We will apply Lemma \ref{lem:KN15-Lemma2} directly to this set $S$.
There are two cases, either $\{ a \} \leq \{ b \}$ or $\{ b \} < \{ a \}$.

First, assume that $0 \leq \{ a \} \leq \{ b \} < 1$.
Then there exists a number $M \in \N$ such that
\[
A_{\geq 1} = A_{\geq 2} = \cdots =A_{\geq M} = [0,\tfrac{1}{N})
\;\supset\;
A_{\geq M+1} = \big[ \tfrac{ \{ a \} }{N} , \tfrac{ \{ b \} }{N} \big)
\;\supset\;
A_{\geq M+2} = \emptyset .
\]
Clearly, we may choose the canonical frequency sets $\Lambda_1 = \cdots = \Lambda_M = N \Z$ for $A_{\geq 1} = A_{\geq 2} = \cdots = A_{\geq M} = [0,\tfrac{1}{N})$, so that for each $n = 1,\ldots,M$, the system $E(\Lambda_n)$ is a Riesz basis (in fact, an orthogonal basis) for $L^2(A_{\geq n})$.
Also, there exists a set $\Lambda_{M+1} \subset N\Z$ such that $E(\Lambda_{M+1})$ is a Riesz basis for $L^2\big[ \frac{ \{ a \} }{N} , \frac{ \{ b \} }{N} \big)$; indeed, such a set $\Lambda_{M+1}$ can be obtained from the result of Seip \cite{Se95} with a dilation (see Lemma \ref{lem:RB-basic-operations}).
Then Lemma \ref{lem:KN15-Lemma2} with shift factors `$n-1$' in place of `$n$' yields that $E( (\cup_{n = 1}^{M} N \Z {+} n {-}1) \cup (\Lambda_{M+1} {+} M) )$ is a Riesz basis for $L^2(S) = L^2( \frac{1}{N} V )$. By a dilation, we obtain that $E( (\cup_{n = 1}^{M} \Z {+} \frac{n {-}1}{N}) \cup (\frac{1}{N} \Lambda_{M+1} {+} \frac{M}{N}) ) = E( \Z \cup (\cup_{k = 1}^{M-1} \Z {+} \frac{k}{N}) \cup (\frac{1}{N} \Lambda_{M+1} {+} \frac{M}{N}) )$ is a Riesz basis for $L^2(V) = L^2 ([0,1) \cup [a,b) )$.
Now, we claim that $E(\Lambda')$ is a Riesz basis for $L^2[a,b)$, where $\Lambda' := ( \cup_{k = 1}^{M-1} \Z {+} \frac{k}{N} ) \cup ( \frac{1}{N} \Lambda_{M+1} {+} \frac{M}{N} )$.
To see this, we again apply Lemma \ref{lem:KN15-Lemma2} (the original version) to the set $S' := \frac{1}{N} V'$ with $V' = [a,b)$.
One can easily check that the corresponding set $A_{\geq n}'$ is equal to the set $A_{\geq n-1}$ above, that is,
\begin{equation}\label{eqn:Ageq-for-subinterval-ab}
A_{\geq 1}' = \cdots =A_{\geq M-1}' = [0,\tfrac{1}{N})
\;\supset\;
A_{\geq M}' = \big[ \tfrac{ \{ a \} }{N} , \tfrac{ \{ b \} }{N} \big)
\;\supset\;
A_{\geq M+1}' = \emptyset .
\end{equation}
Then Lemma \ref{lem:KN15-Lemma2} implies that $E( (\cup_{k = 1}^{M-1} \Z {+} \frac{k}{N}) \cup (\frac{1}{N} \Lambda_{M+1} {+} \frac{M}{N}) )$ is a Riesz basis for $L^2[a,b)$, as claimed.

Now, assume that $0 \leq \{ b \} < \{ a \} < 1$.
Then there exists a number $M \in \N$ such that
\[
A_{\geq 1} = \cdots =A_{\geq M} = [0,\tfrac{1}{N})
\;\supset\;
A_{\geq M+1} = \big[ 0, \tfrac{ \{ b \} }{N} \big) \cup \big[ \tfrac{ \{ a \} }{N}, \tfrac{1}{N} \big)
\;\supset\;
A_{\geq M+2} = \emptyset .
\]
Again, using the result of Seip \cite{Se95} with a dilation, we obtain a set $\Lambda_{M+1} \subset N\Z$ such that $E(\Lambda_{M+1})$ is a Riesz basis for $L^2\big[ \frac{ \{ a \} }{N} , \frac{ 1+\{ b \} }{N} \big)$. Since all elements in $E(N\Z)$ are $\frac{1}{N}$-periodic, it follows that $E(\Lambda_{M+1})$ is a Riesz basis for $L^2 \big( \big[ 0, \frac{ \{ b \} }{N} \big) \cup \big[ \frac{ \{ a \} }{N}, \frac{1}{N} \big) \big)$. Then, by similar arguments as in the case $\{ a \} \leq \{ b \}$, we deduce that $E( \Z \cup (\cup_{k = 1}^{M-1} \Z {+} \frac{k}{N}) \cup (\frac{1}{N} \Lambda_{M+1} {+} \frac{M}{N}) )$ is a Riesz basis for $L^2 ( [0,1) \cup [a,b) )$, and that
$E( (\cup_{k = 1}^{M-1} \Z {+} \frac{k}{N}) \cup (\frac{1}{N} \Lambda_{M+1} {+} \frac{M}{N}) )$ is a Riesz basis for $L^2[a,b)$.

\medskip

\noindent
\textbf{Case $L \geq 2$}. \
We will use essentially the same arguments as in the case $L=1$, but employ the main result of Kozma and Nitzan \cite{KN15} instead of Seip \cite{Se95}.
Given a set $V = [0,1) \cup [a_1,b_1) \cup \cdots \cup [a_L,b_L) \subset [0,N)$ with $L, N \in \N$ and $1 \leq a_1 < b_1 < \cdots < a_L < b_L \leq N$, let $S := \frac{1}{N} V = [0,\frac{1}{N}) \cup [\frac{a_1}{N},\frac{b_1}{N}) \cup \cdots \cup [\frac{a_L}{N},\frac{b_L}{N}) \subset [0,1)$.
As in the case $L=1$, we will apply Lemma \ref{lem:KN15-Lemma2} to this set $S$.

Note that there are finitely many possible ordering of the values
\[
0
\;\leq\;
\tfrac{ \{ a_1 \} }{N} , \tfrac{ \{ b_1 \} }{N} , \cdots , \tfrac{ \{ a_L \} }{N}, \tfrac{ \{ b_L \} }{N}
\;<\;
\tfrac{1}{N} ,
\]
where equalities are also allowed, e.g., the values are all zero if all $a_\ell$ and $b_\ell$ are integers.
It is easily seen that besides $0$ and $\tfrac{1}{N}$, these are the only possible values that can be the boundary points of $A_{\geq n}$, $n=1, \ldots, N$.
In any case, since $[0,\frac{1}{N}) \subset S$ we have
\[
A_{\geq 1} = [0,\tfrac{1}{N})
\;\supset\; A_{\geq 2}
\;\supset\; \cdots
\;\supset\; A_{\geq N} ,
\]
where each of the sets $A_{\geq 2} , \ldots, A_{\geq N}$ is either empty or a finite union of intervals.
One can therefore use the main result of \cite{KN15} with a dilation, to construct sets $\Lambda_1 {=} N\Z, \Lambda_2, \Lambda_3 , \ldots, \Lambda_N \subset N\Z$ such that for each $n$, the system $E(\Lambda_n)$ is a Riesz basis for $L^2(A_{\geq n})$.
The rest of the proof is similar to the case $L=1$.
\hfill $\Box$
%%%\hfill $\square$ %end of proof

\section*{Acknowledgments}
A.~Caragea acknowledges support by the DFG (German Research Foundation) Grant PF 450/11-1.
D.G.~Lee acknowledges support by the DFG Grants PF 450/6-1 and PF 450/9-1.
The authors would like to thank
Carlos Cabrelli, Diana Carbajal, and Felix Voigtlaender for valuable comments concerning the linear independence over $\Q$.
The names of authors are ordered alphabetically by convention, and both authors contributed equally to this work.
%%%%% The authors would like to thank Goetz E.~Pfander for fruitful discussions which led to this work.

\appendix

\renewcommand{\thetheorem}{A.\arabic{theorem}}
\setcounter{theorem}{0}

\section{Proof of Lemma \ref{lem:KN15-Lemma2-variant-Nprime}}
\label{appendix:proof-or-KN15-Lemma2-variant-Nprime}

To prove Lemma \ref{lem:KN15-Lemma2-variant-Nprime}, we will follow the proof strategy of Lemma \ref{lem:KN15-Lemma2} \cite[Lemma 2]{KN15}.
For any $N \in \N$, a measurable set $S \subset [0,1)$, and $n = 0,1, \ldots, N$, we define
\[
\begin{split}
& A_n
:= \Big\{ t \in [0, \tfrac{1}{N}) : t+ \tfrac{k}{N} \in S \;\; \text{for exactly $n$ values of} \;\; k \in \{ 0 , 1, \ldots, N-1 \} \Big\} , \\
& B_n
:= \Big\{ t \in S : t+ \tfrac{k}{N} \in S \;\; \text{for exactly $n$ integers} \;\; k \in \Z \Big\} .
\end{split}
\]
Obviously, considering the set $B_n$ modulo $\frac{1}{N}$ yields the \emph{$n$-fold} of $A_n$, which means that each element of $A_n$ corresponds to exactly $n$ points of $B_n$ that are distanced apart by multiples of $\frac{1}{N}$.
Note that $\{ A_n \}_{n=0}^N$ and $\{ B_n \}_{n=0}^N$ form partitions of $[0, \frac{1}{N})$ and $[0,1)$, respectively, that is, $[0, \frac{1}{N}) = \cupdot_{n=0}^N \, A_n$ and $[0,1) = \cupdot_{n=0}^N \, B_n$. Also, the family $\{ B_n \}_{n=1}^N$ forms a partition of $S$, i.e., $S = \cupdot_{n=1}^N \, B_n$.
For $f \in L^2(S)$ and $n = 1, \ldots, N$, we denote by $f_n$ the restriction of $f$ to $B_n$, that is, $f_n(t) = f(t)$ for $t \in B_n$ and $0$ otherwise. This yields the decomposition $L^2(S) \ni f = f_1 + \ldots + f_N$ with all $f_n$ having disjoint support.
Note that the set $A_{\geq n}$ given by \eqref{eqn:A-geq-n} can be expressed as $A_{\geq n} = \cup_{\ell=n}^N \, A_{\ell}$ for $n = 1, \ldots, N$.
Similarly, we define $B_{\geq n} := \cup_{\ell=n}^N \, B_{\ell}$ and $f_{\geq n} := \sum_{\ell=n}^N f_{\ell}$ for $n = 1, \ldots, N$.
For brevity, we write $\Lambda := \cup_{\ell=1}^N (\Lambda_{\ell} {+} j_{\ell})$.

\medskip

\noindent
\textbf{Frame}. \
Assume that $\Lambda_1, \ldots, \Lambda_N \subset N\Z$ are such that $E(\Lambda_n)$ is a frame for $L^2(A_{\geq n})$.
To prove that $E(\Lambda)$ is a frame for $L^2(S)$, it is enough to show that there exists a constant $c > 0$ satisfying
\[
\begin{split}
&\sum_{\lambda \in \Lambda} \big| \langle f, e^{2 \pi i \lambda (\cdot)} \rangle_{L^2(S)} \big|^2
\;\geq\; c \, \| f_n \|_{L^2(S)}^2 - \sum_{\ell=1}^{n-1} \| f_\ell \|_{L^2(S)}^2  \\
&\quad \text{for all} \;\; f \in L^2(S) \;\; \text{and} \;\; n = 1,\ldots,N .
\end{split}
\]
In turn, it is enough to show that there exists a constant $c > 0$ satisfying
\begin{equation}\label{eqn:proofFR-claim}
\sum_{\lambda \in \Lambda} \big| \langle f_{\geq n}, e^{2 \pi i \lambda (\cdot)} \rangle_{L^2(S)} \big|^2
\;\geq\; c \, \| f_n \|_{L^2(S)}^2 ,
\quad f \in L^2(S) , \;\; n = 1,\ldots,N .
\end{equation}
Such reductions are essentially due to the decomposition $f = f_1 + \ldots + f_N$ with all $f_n$ having disjoint support, and due to $S \subset [0,1)$ and $\Lambda \subset \Z$; see \cite[Eqns.~(6)--(7)]{KN15} for detailed arguments.

To prove \eqref{eqn:proofFR-claim}, fix any $f \in L^2(S)$ and any $n \in \{ 1, \ldots, N \}$.
Since $f_{\geq n}$ is supported in $\cup_{\ell=n}^N \, B_{\ell} \subset S$, we have for any $\lambda \in \Lambda_{\ell} {+} j_{\ell}$ with $\ell \in \{ 1, \ldots, N \}$,
\[
\begin{split}
\big\langle f_{\geq n}, e^{2 \pi i \lambda (\cdot)} \big\rangle_{L^2(S)}
&= \int_0^1  f_{\geq n}(t) \, e^{- 2 \pi i \lambda t} \, dt \\
&= \int_0^{1/N}  \sum_{k=0}^{N-1} f_{\geq n}\big(t+\tfrac{k}{N}\big) \, \exp (- 2 \pi i \lambda (t+\tfrac{k}{N}) ) \, dt \\
&= \int_0^{1/N}  h_{n,\ell} (t) \, e^{- 2 \pi i \lambda t} \, dt
= \big\langle h_{n,\ell}, e^{2 \pi i \lambda (\cdot)} \big\rangle_{L^2[0,\frac{1}{N})} \;,
\end{split}
\]
where
\begin{equation}\label{eqn:def-h-n-ell}
h_{n,\ell} (t) := \mathds{1}_{A_{\geq n}} (t) \,\cdot \sum_{k=0}^{N-1} f_{\geq n}\big(t+\tfrac{k}{N}\big) \, e^{- 2 \pi i j_{\ell} k/N} .
\end{equation}
Note that for $\ell = 1,\ldots,n$, the function $h_{n,\ell}$ is supported in $A_{\geq n} \subset A_{\geq \ell}$. Since $E(\Lambda_\ell)$ is a frame for $L^2(A_{\geq \ell})$, say, with lower frame bound $\alpha_\ell > 0$, we have
\begin{equation}\label{eqn:proofFR-sum-lambda-in-Lambdajell}
\begin{split}
\sum_{\lambda \in \Lambda_{\ell} + j_{\ell}} \big| \langle f_{\geq n}, e^{2 \pi i \lambda (\cdot)} \rangle_{L^2(S)} \big|^2
&\;=\; \sum_{\lambda \in \Lambda_{\ell} + j_{\ell}} \big| \langle h_{n,\ell}, e^{2 \pi i \lambda (\cdot)} \rangle_{L^2[0,\frac{1}{N})} \big|^2 \\
&\;\geq\; \alpha_\ell \, \| h_{n,\ell} \|^2 .
\end{split}
\end{equation}
Summing up \eqref{eqn:proofFR-sum-lambda-in-Lambdajell} for $\ell = 1,\ldots,n$ gives
\begin{equation}\label{eqn:proofFR-sumup-over-ell-in-Ln}
\begin{split}
\sum_{\lambda \in \Lambda} \big| \langle f_{\geq n}, e^{2 \pi i \lambda (\cdot)} \rangle_{L^2(S)} \big|^2
&\;\geq\;
\sum_{\ell=1}^n \sum_{\lambda \in \Lambda_{\ell} + j_{\ell}} \big| \langle f_{\geq n}, e^{2 \pi i \lambda (\cdot)} \rangle_{L^2(S)} \big|^2 \\
&\;\geq\;
\Big(\min_{1 \leq \ell \leq n}  \alpha_\ell \Big) \cdot \sum_{\ell=1}^n \| h_{n,\ell} \|^2 \\
&\;\geq\;
\Big(\min_{1 \leq \ell \leq n}  \alpha_\ell \Big) \cdot \sum_{\ell=1}^n \| h_{n,\ell} \cdot \mathds{1}_{A_n} \|^2 .
\end{split}
\end{equation}
On the other hand, for any fixed $t \in A_n$, Equation~\eqref{eqn:def-h-n-ell} becomes
\[
h_{n,\ell} (t) = \sum_{k=0}^{N-1} f_{\geq n}\big(t+\tfrac{k}{N}\big) \, e^{- 2 \pi i j_{\ell} k/N}
\]
and collecting the equation for $\ell = 1,\ldots,n$ gives the $n {\times} N$ linear system
\[
\begin{bmatrix}
h_{n,\ell} (t)
\end{bmatrix}_{\ell=1}^n
=
\begin{bmatrix}
e^{- 2 \pi i j_{\ell} k/N}
\end{bmatrix}_{1 \leq \ell \leq n, 0 \leq k \leq N-1}
\begin{bmatrix}
f_{\geq n}\big(t+\tfrac{k}{N}\big)
\end{bmatrix}_{k=0}^{N-1} .
\]
Since $t \in A_n$, the vector $[ f_{\geq n}\big(t+\tfrac{k}{N}\big)
]_{0 \leq k \leq N-1}$ has exactly $n$ nonzero entries, say, at the indices $k_1 < \ldots < k_n$ from $\{0,1,\ldots,N-1\}$. This reduces the above system to an $n {\times} n$ linear system
\[
\begin{bmatrix}
h_{n,\ell} (t)
\end{bmatrix}_{\ell=1}^n
=
\begin{bmatrix}
e^{- 2 \pi i j_{\ell} k_r/N}
\end{bmatrix}_{1 \leq \ell \leq n, 1 \leq r \leq n}
\begin{bmatrix}
f_{\geq n}\big(t+\tfrac{k_r}{N}\big)
\end{bmatrix}_{r=1}^n ,
\]
where the associated matrix $[ e^{- 2 \pi i j_{\ell} k_r/N} ]_{1 \leq \ell \leq n, 1 \leq r \leq n}$
is invertible since $N$ is prime (by Chebotar\"{e}v's theorem on roots of unity, see e.g., \cite{SL96}).
As there are only finitely many possible choices of $k_1 < \ldots < k_n$ from $\{0,1,\ldots,N-1\}$, there exists a constant $c' > 0$ such that
\[
\sum_{\ell=1}^n  \big| h_{n,\ell} (t) \big|^2
\;\geq\;
c' \,
\sum_{k=0}^{N-1}  \big| f_{\geq n}\big(t+\tfrac{k}{N}\big) \big|^2
\quad \text{for all} \;\; t \in A_n .
\]
Integrating over $t \in A_n$ then gives
\[
\sum_{\ell=1}^n \| h_{n,\ell} \cdot \mathds{1}_{A_n} \|^2
\geq
c'
\int_{A_n} \sum_{k=0}^{N-1}  \big| f_{\geq n}\big(t+\tfrac{k}{N}\big) \big|^2
=
c' \, \| f_n \|_{L^2(B_n)}^2
=
c' \, \| f_n \|_{L^2(S)}^2 .
\]
Combining this inequality with \eqref{eqn:proofFR-sumup-over-ell-in-Ln} yields the desired inequality \eqref{eqn:proofFR-claim}.

\medskip

\noindent
\textbf{Riesz sequence}. \
Assume that $\Lambda_1, \ldots, \Lambda_N \subset N\Z$ are such that $E(\Lambda_n)$ is a Riesz sequence in $L^2(A_{\geq n})$.
We will show that $E(\Lambda)$ is a Riesz sequence in $L^2(S)$ by using the \emph{frame} part which is proved above (a similar trick was used in \cite[Lemma 7]{KN16}).
Let $S' := [0,1) \backslash S$ and let $A_{\geq n}'$, $n=1,\ldots,N$, be the corresponding sets of \eqref{eqn:A-geq-n} for $S'$.
It is easily seen that $A_{\geq n}' = [0,\frac{1}{N}) \backslash A_{\geq N+1-n}$ for $n=1,\ldots,N$.
Since $E(\Lambda_{N+1-n})$ is a Riesz sequence in $L^2(A_{\geq N+1-n})$, we deduce from Proposition \ref{prop:Prop5-4-BCMS19} with a dilation that the system $E(N\Z \backslash \Lambda_{N+1-n})$ is a frame for $L^2 ( [0,\frac{1}{N}) \backslash A_{\geq N+1-n} ) = L^2(A_{\geq n}')$; see the discussion after Proposition \ref{prop:Prop5-4-BCMS19}.
The \emph{frame} part then implies that the system
$E( \cup_{n=1}^N ( (N\Z \backslash \Lambda_{N+1-n}) {+} j_{N+1-n} ) )$ is a frame for $L^2 (S')$.
Finally, again by Proposition \ref{prop:Prop5-4-BCMS19}, we conclude that
the system $E(\Lambda) = E( \cup_{n=1}^N (\Lambda_n {+} j_n) ) = E( \Z \backslash \cup_{n=1}^N ( (N\Z \backslash \Lambda_{N+1-n}) {+} j_{N+1-n} ) )$ is a Riesz sequence in $L^2(S) = L^2 ( [0,1) \backslash S' )$.

\medskip

\noindent
\textbf{Riesz basis}. \
Since a family of vectors in a separable Hilbert space is a Riesz basis if and only if it is both a frame and a Riesz sequence, this part follows immediately by combining the \emph{frame} and \emph{Riesz sequence} parts.
\hfill $\Box$
%%%\hfill $\square$ %end of proof

\end{document}